\documentclass[a4paper,10pt]{article}
\usepackage[utf8]{inputenc}
\usepackage{amsmath,amsfonts}
\usepackage{amsthm}
\usepackage{amssymb}
\usepackage{multirow}
\usepackage[backref=page]{hyperref}
\newtheorem{teo}{Theorem}[section]
\newtheorem{lem}[teo]{Lemma}
\newtheorem{prop}[teo]{Proposition}
\newtheorem{cor}[teo]{Corollary}
\newtheorem*{teo*}{Theorem}

\newtheorem{cnj}{Conjecture}
\newtheorem*{cnj*}{Conjecture}
\theoremstyle{definition}
\newtheorem{dfn}[teo]{Definition}
\newtheorem*{dfn*}{Definition}

\newtheorem{oss}[teo]{Remark}
\newtheorem*{dom*}{Question}

\newtheorem*{cnj:stok}{Stoker Conjecture}
\newtheorem*{cnj:wstok}{Weak Stoker Conjecture}
\newtheorem*{cnj:volconj}{The Maximum Volume Conjecture}
\newcommand{\ra}{\rightarrow}
\newcommand{\R}{\mathbb{R}}
\newcommand{\vol}{\textrm{Vol}}

\newcommand{\Pg}{\mathcal{P}_\Gamma}
\newcommand{\h}{\mathbb{H}}
\title{The volume conjecture for polyhedra implies the Stoker conjecture}
\author{Giulio Belletti}
\date{}
\newcommand{\address}{{
  \bigskip
  \footnotesize

  Giulio Belletti, \textsc{Universit\'e Paris-Saclay}\par\nopagebreak
  \textit{E-mail address},  \texttt{giulio.belletti@universite-paris-saclay.fr}

}}
\begin{document}

\maketitle

\begin{abstract}
 We show that the Volume Conjecture for polyhedra implies a weak version of the Stoker Conjecture; in turn we prove that this weak version of the Stoker conjecture implies the Stoker conjecture. The main tool used is an extension of a result of Montcouquiol and Weiss, saying that dihedral angles are local coordinates for compact polyhedra with angles $\leq \pi$.
\end{abstract}

\section{Introduction}

To what extent dihedral angles determine a hyperbolic polyhedron is a long standing question, first posed in \cite[Page 152]{stoker}. 

\begin{cnj*}[The Stoker conjecture]
Let $P_1$ and $P_2$ be two hyperbolic polyhedra with $1$-skeleton $\Gamma$ and the same dihedral angle at each corresponding edge. Then $P_1$ and $P_2$ are isometric.
\end{cnj*}

Over the next half century there have been several results in this direction: as soon as some conditions are required of the polyhedra (for example, acute-angled in \cite{andreev}, simple in \cite{rivhodg} or hyperideal in \cite{bonbao}) then the conjecture is known to be true. Moreover \cite{mont} and \cite{weiss} independently proved that the conjecture is true locally, i.e. the dihedral angles are local coordinates for the space of polyhedra with fixed $1$-skeleton.

In this paper, we propose a weak version of the Stoker conjecture:

\begin{cnj*}[Weak Stoker conjecture]
Let $P_1$ and $P_2$ be two polyhedra with $1$-skeleton $\Gamma$ and the same dihedral angle at each corresponding edge. Then $P_1$ and $P_2$ have the same volume.
\end{cnj*}

Clearly the Stoker conjecture implies the weak Stoker Conjecture; in Corollary \ref{cor:equiv} we prove that in fact the converse also holds.

The interest of the weak Stoker Conjecture is two-fold: firstly, it could be a more tractable problem than the Stoker conjecture; secondly, it is actually a consequence of the Volume Conjecture for polyhedra, providing a connection between a long standing open problem in hyperbolic geometry and quantum invariants. 

Since the introduction of quantum invariants in \cite{witten}, many conjectures have been proposed to link their asymptotic behavior to several topological and geometric properties. Among them the subject of the most research has been the plethora of volume conjectures, relating the growth of quantum invariants of objects to the volume of hyperbolic structures on them. The first such conjecture was proposed by Kashaev (\cite{kash}, see also \cite{murmur}) relating the growth of the Jones polynomial of a knot evaluated at a certain root of unity to the volume of the complement of the knot. More recently, Chen and Yang \cite{cyvolconj} proposed a volume conjecture for the Turaev-Viro invariant of compact manifolds. Both of these conjectures have extensions relating higher order terms of the asymptotic expansion of a quantum invariant to different quantities such as the Reidemeister torsion. These conjectures have been proven for certain families of knots or manifolds (see for example \cite{ohtdf} and \cite{BDKY2018} for results about the Volume Conjecture of Chen and Yang, and \cite{ohtsuki52},\cite{ohtsuki6}, \cite{ohtsuki7} for results on the Kashaev volume conjecture).

A related series of conjectures connects the quantum invariants of a planar graph $\Gamma$ in $S^3$ to the volume of hyperbolic polyhedra that have $\Gamma$ as a $1$-skeleton. The first such conjecture was given in \cite{volconjpoly} for trivalent planar graphs and simple hyperideal polyhedra; it was extended in \cite{murkolp} and later in \cite{maxvolconj} to any polyhedral (i.e. planar and $3$-connected) graph and any hyperbolic polyhedron.

\begin{cnj*}
 [The volume conjecture for polyhedra]
Let $P$ be a hyperbolic polyhedron with dihedral angles
  $\alpha_1,\dots,\alpha_m$ at the edges $e_1,\dots,e_m$, and $1$-skeleton $\Gamma$. Let $col_r$ be a sequence of $r$-admissible colorings of the edges $e_1,\dots,e_m$ of $\Gamma$ such that 
  \begin{displaymath}
   2\pi\lim_{r\ra+\infty}\frac{col_r(e_i)}{r}=\pi-\alpha_i.
  \end{displaymath}
Then
\begin{displaymath}
 \lim_{r\ra+\infty}\frac{\pi}{r}\log\left\lvert Y_r(S^3,\Gamma,col_r,e^{2\pi i/r})\right\rvert=\mathrm{Vol}(P).
\end{displaymath}
where $Y_r\left(S^3,\Gamma,col,e^{2\pi i/r}\right)$ is the Yokota invariant of the graph $\Gamma$ with edges colored with $col_r$ evaluated at $e^{2\pi i/r}$ (see Section \ref{sec:quantum} and \cite[Section 2]{maxvolconj} for more details and definitions).
\end{cnj*}

In Section \ref{sec:hyppol} we give an overview of hyperbolic polyhedra and we prove the technical results about the space of hyperbolic polyhedra with fixed $1$-skeleton that are needed in the proof of the main result. In Section \ref{sec:quantum} we give a brief review about the Yokota invariant and the Volume conjecture for polyhedra. Finally in Section \ref{sec:main} we prove the equivalence of the Weak Stoker Conjecture and the Stoker Conjecture. Quantum invariants are only involved in Section \ref{sec:quantum}; a reader only interested in the geometric results of this paper can safely skip it if so desired.

\section{Hyperbolic polyhedra}\label{sec:hyppol}

\begin{dfn}\label{dfn:poly}
 A \emph{hyperbolic polyhedron} is a compact intersection of half spaces of $\mathbb{H}^3$.
\end{dfn}

\begin{oss}
 Throughout the literature, other types of hyperbolic polyhedra have also been considered, for example polyhedra with ideal or hyperideal vertices. The results of this paper can be easily extended to \emph{proper generalized hyperbolic polyhedra} (see \cite{maxvol}) with no ideal vertices, by simply truncating them to obtain a hyperbolic polyhedron of Definition \ref{dfn:poly}. However the techniques here presented cannot be directly applied to the case of polyhedra with some ideal vertices, since it is unknown in this case whether angles are local coordinates; the Stoker conjecture for polyhedra with only ideal vertices is known to be true by \cite{rivin}. 
\end{oss}

\begin{dfn}
 If $\Gamma$ is a graph we say that a polyhedron $P$ has $1$-skeleton $\Gamma$ if there exists a simplicial map $\phi:\Gamma\ra P$ sending vertices to vertices injectively and edges to edges.
\end{dfn}

\begin{oss}
 A graph $\Gamma$ is the $1$-skeleton of a polyhedron if and only if it is planar and $3$-connected (that is to say, it has more than three vertices and cannot be disconnected by removing two edges) \cite{steinitz}. Such a graph admits a unique embedding in $S^2$ up to planar isotopy and reflections \cite[Corollary 3.4]{fle}; therefore, even though $\Gamma$ is just an abstract graph, it is convenient to treat it as an embedded graph in the sphere. As such, we will often speak of its faces, meaning the faces of the cellularization of $S^2$ induced by an embedding.
\end{oss}

\begin{cnj}[The Stoker conjecture]\label{cnj:stok}
 Let $P_1$ and $P_2$ be two polyhedra with the same $1$-skeleton $\Gamma$. If for every edge $e$ of $\Gamma$, the dihedral angles of $P_1$ and $P_2$ at their corresponding edge are equal, then $P_1$ and $P_2$ are isometric.
\end{cnj}

\begin{cnj}[A weak version of the Stoker conjecture]\label{cnj:wstok}
 Let $P_1$ and $P_2$ be two polyhedra with the same $1$-skeleton $\Gamma$. If for every edge $e$ of $\Gamma$, the dihedral angles of $P_1$ and $P_2$ at their corresponding edge are equal, then $P_1$ and $P_2$ have the same volume.
\end{cnj}

\begin{teo}
 Conjecture \ref{cnj:stok} is equivalent to Conjecture \ref{cnj:wstok}.
\end{teo}

\begin{teo}\label{thm:main}
 If Conjecture \ref{cnj:volconj} is true, then the Stoker conjecture is true.
\end{teo}

To prove Theorem \ref{thm:main} we will need to take a polyhedron $P$ with $1$-skeleton $\Gamma$, choose a diagonal $e$ of a non-triangular face, and deform $P$ to make $e$ into an edge. To do this, it is convenient to treat $P$ as a polyhedron with $1$-skeleton equal to $\Gamma\cup e$; thus we introduce a concept of \emph{weak} $1$-skeleta.

\begin{dfn}
Let $e\subseteq \Gamma$ be an edge. A polyhedron $P$ has \emph{weak $1$-skeleton} $(\Gamma,e)$ if there exists a simplicial map $\phi:\Gamma\ra P$ sending vertices to vertices injectively, edges different from $e$ to edges of $P$ and the edge $e$ to a face of $P$.
\end{dfn}

Put in other terms, in a polyhedron with weak $1$-skeleton $(\Gamma,e)$ we allow the two faces intersecting in $e$ to lie in the same plane. If this happens, the $1$-skeleton of $P$ is actually the graph obtained from $\Gamma$ by removing the edge $e$; otherwise, the $1$ skeleton of $P$ is simply $\Gamma$. Notice that if the $1$-skeleton of $P$ is $\Gamma$, then $P$ has weak skeleton $(\Gamma,e)$ for any edge $e$.

\begin{dfn}
 The space of polyhedra with $1$-skeleton $\Gamma$ is denoted $\mathcal{P}_\Gamma$. The space of polyhedra with weak $1$-skeleton $(\Gamma,e)$ is denoted $\mathcal{P}_\Gamma^e$. 
\end{dfn}
\begin{oss}
 A polyhedron with $F$ faces is an intersection of $F$ half spaces. The set $\mathcal{H}$ of half spaces in $\R^3$ is naturally a manifold; thus both $\mathcal{P}_\Gamma$  and $\mathcal{P}_\Gamma^e$ are naturally subsets of the manifold $\mathcal{H}^F$. Thus they are topological spaces; we will prove next that they are actually smooth manifolds.
\end{oss}

Let $\Gamma'$ be any $3$-connected planar graph and $\Gamma$ be the graph obtained from $\Gamma'$ by adding an edge $e$ to a non-triangular face of $\Gamma'$.

\begin{prop}
 The space $\mathcal{P}_\Gamma$ is a manifold; the space $\mathcal{P}_\Gamma^e$ is a manifold with boundary whose interior is $\mathcal{P}_\Gamma$ and whose boundary is $\mathcal{P}_{\Gamma'}$.
\end{prop}
\begin{proof}
The proof for $\mathcal{P}_\Gamma$ is in \cite[Proposition 17]{mont}; however since the proof for $\mathcal{P}_\Gamma^e$ is adapted from it we briefly sketch it.
 The space $\mathcal{P}_\Gamma$ is an open set of the space of Euclidean polyhedra with $1$-skeleton $\Gamma$ (namely, the Euclidean polyhedra with every vertex contained in the open unit ball); therefore if we prove that the latter is a manifold we are done.
 
 Denote with $F=\{f_1,\dots,f_n\}$ the set of faces of $\Gamma$, and with $\mathcal{H}$ the set of half spaces in $\R^3$. Then $\mathcal{P}_\Gamma$ is a subset of $\mathcal{H}^F$; the inclusion map sends $P\in\mathcal{P}_\Gamma$ to the $F$-tuple $\left(\Pi_{f_1},\dots,\Pi_{f_n}\right)$ of half-spaces supporting the faces of $P$.
 
The image of this inclusion is characterized as the tuples $\left(\Pi_{f_1},\dots,\Pi_{f_n}\right)\in\mathcal{H}^F$ such that:
 \begin{enumerate}
  \item for every $v\in \Gamma$ lying on the faces $f_{i_1},\dots,f_{i_k}$, the corresponding half planes $\Pi_{f_{i_1}},\dots,\Pi_{f_{i_k}}$ intersect in a single point $v_{i_1,\dots,i_k}\in\R^3$;\label{cond:1}
  \item if $j\notin\{i_1,\dots,i_k\}$, then $v_{i_1,\dots,i_k}$ is in the interior of $\Pi_j$.\label{cond:2}
 \end{enumerate}
 The first condition is open if $v$ is a trivalent vertex and closed otherwise.
 The second condition is clearly open; in particular it implies that the intersection defining a polyhedron must be non-redundant (i.e. if we remove any half space, the polyhedron changes) and that the intersection of half-planes around different vertices are different points of $\h^3$. 
 
 These conditions can also be described by equations. The first condition is equivalent to the set of equations
 
 $$\Psi_{v,j}=\mathrm{det}\left(\partial\Pi_{i_j}^*,\partial\Pi_{i_{j+1}}^*,\partial\Pi_{i_{j+2}}^*,\partial\Pi_{i_{j+3}}^*\right)=0$$
 
 for $j\in \{1,\dots,k-3\}$. Put in other words, this equation asks that the dual vectors to the boundaries  of $\Pi_{i_j},\dots,\Pi_{i_{j+3}}$ must be coplanar, which is equivalent to asking that $\partial\Pi_{i_j},\dots,\partial\Pi_{i_{j+3}}$ intersect in a single point. In \cite[Proposition 17]{mont} it is proven that these conditions are linearly independent.
 
 To describe the second condition as an equation we define the \emph{signed distance} of $v\in\R^3$ from a half space $H$, denoted $\delta(v,H)$, as the real number having absolute value the Euclidean distance from $v$ to $\partial H$, positive sign if $v\in H$ and negative sign otherwise. Then condition \ref{cond:2} is 
 
 $$\delta(v_{i_1,\dots,i_k}, \Pi_j)>0$$
 for all $j\notin{i_1,\dots,i_k}$. Since these conditions are open, then the space $\mathcal{P}_\Gamma$ is a manifold.

 Consider now the polyhedra with weak $1$-skeleton $(\Gamma,e)$, and suppose (after reindexing if necessary) that $f_1$ and $f_2$ are the two faces of $\Gamma$ adjacent to $e$. Then $\mathcal{P}_\Gamma^e$ is a subset of $\mathcal{H}^F$ as before, and the tuples $\left(\Pi_{f_1},\dots,\Pi_{f_n}\right)\in\mathcal{P}_\Gamma^e$ are characterized by the following conditions: 
 \begin{enumerate}
   \item for every $v\in \Gamma$ lying on the faces $f_{i_1},\dots,f_{i_k}$, the corresponding half planes $\Pi_{f_{i_1}},\dots,\Pi_{f_{i_k}}$ intersect in a single point $v_{i_1,\dots,i_k}\in\R^3$;\label{cond:1'}
  \item if $j\notin\{1,2,i_1,\dots,i_k\}$, then $v_{i_1,\dots,i_k}$ is in the interior of $\Pi_j$;\label{cond:2'}
  \item if $j=1\notin \{i_1,\dots,i_k\}$ and $2\notin \{i_1,\dots,i_k\}$ or viceversa, then $v_{i_1,\dots,i_k}$ is in the interior of $\Pi_j$;\label{cond:3'}
  \item if $j=1\notin\{i_1,\dots,i_k\}$ and $2\in\{i_1,\dots,i_k\}$ or viceversa, then $v_{i_1,\dots,i_k}$ belongs to $\Pi_j$.\label{cond:4'}
 \end{enumerate}

 The first and fourth conditions are closed while the second and third are open. We can describe these conditions as equations in a similar manner as before; the first conditions carry over without change. The second condition becomes
  $$\delta(v_{i_1,\dots,i_k}, \Pi_j)>0$$
 for all $j\notin{1,2,i_1,\dots,i_k}$; the third condition is analogous. The fourth condition is instead
 
 $$\delta(v_{i_1,\dots,i_k}, \Pi_j)\geq0$$
 if $j=1\notin\{i_1,\dots,i_k\}$ and $2\in\{i_1,\dots,i_k\}$ or viceversa.
 
 To prove that $\mathcal{P}_\Gamma^e$ is a manifold, we need to show that the equations of condition \ref{cond:1'} and the equalities in condition \ref{cond:4'} are independent.
 
 Now notice that if $\delta(v_{i_1,\dots,i_k},\Pi_1)=0$, then $v_{i_1,\dots,i_k}$ belongs to $\partial\Pi_1$. This can only happen if $2\in\{i_1,\dots,i_k\}$ by the aforementioned conditions. Since $\partial \Pi_1$ and $\partial \Pi_2$ share three distinct points (the two vertices that are endpoints of $e$ and $v_{i_1,\dots,i_k}$) they must coincide; therefore this condition is equivalent to $\Pi_1=\Pi_2$. 
 
 Suppose then that there is a linear dependence for the equations $\Psi_{v,i}=0$ and $\Pi_1=\Pi_2$; this would imply that there is a linear dependence for the analogous system of equations in $\mathcal{P}_{\Gamma'}$ which is false.
 
 Therefore $\mathcal{P}_\Gamma^e$ is a smooth manifold whose boundary is $\mathcal{P}_{\Gamma'}$ and whose interior is $\mathcal{P}_\Gamma$.

\end{proof}
\begin{cor}\label{cor:manifold}
 The space $\mathcal{A}_\Gamma$ is a manifold; the space $\mathcal{A}_\Gamma^e$ is a manifold with boundary whose interior is $\mathcal{A}_\Gamma$ and whose boundary is $\mathcal{A}_{\Gamma'}$. 
\end{cor}
\begin{proof}
 We need to show that the group $\textrm{Iso}(\h^3)$ acts on $\Pg$ freely and properly. To see that the action is free notice that since the polyhedra in $\Pg$ are non degenerate, an isometry preserving $P\in\Pg$ must fix at least $4$ points in generic position (i.e. not all contained in a plane) and thus must be the identity. To prove that the action is proper we can repeat verbatim the proof of \cite[Lemma 8]{bonbao}. Notice that the proof in \cite{bonbao} only requires that the polyhedra considered have $4$ faces in generic position. Therefore the same argument carries over unchanged for $\Pg^e$.
\end{proof}

The following lemma essentially states that for any polyhedron $P$ and any diagonal $e$, we can deform $P$ to ``add'' $e$ as an edge, increasing the angle at $e$ with a non-zero derivative.

\begin{lem}\label{lem:mainlem}
 For any $P\in\partial\Pg^e$ there exists an $\epsilon>0$ and a family $P_t$ for $t\in[0,\epsilon)$ such that
 \begin{itemize}
  \item $P_0=P$;
  \item $P_t$ is in the interior of $\Pg^e$ for any $t>0$;
  \item the dihedral angle of $P_t$ at $e$ is equal to a certain smooth function $f(t)$ with non-vanishing derivative at $0$.
 \end{itemize}
\end{lem}
\begin{proof}
 Since $P$ is in $\partial\Pg^e$ then the two faces $F_1$ and $F_2$ intersecting in $e$ actually coincide. As before name $\Gamma'$ the graph obtained from $\Gamma$ by removing $e$. Consider $P\subseteq\h^3\subseteq\R^3$ and view it as a Euclidean polyhedron. The dual polyhedron $P^*$ is a Euclidean polyhedron with $1$-skeleton $\Gamma'^*$ which is contained in a ball of radius $r$ centered in $0$.
 Notice that $P^*$ has a vertex $V$ which is dual to $F_1=F_2$. The graph $\Gamma'^*$ is obtained from $\Gamma^*$ by shrinking the edge $e^*$ down to a point which we denote with $v$.
 
 Consider $\Phi$ the homothety centered in $0$ of factor $\frac{1}{2r}$; then $\Phi(P^*)$ is a polyhedron contained in a ball centered in $0$ of radius $\frac{1}{2}$ and thus we can consider it as a hyperbolic polyhedron. By \cite[Theorem 1.2 and Section 4.4]{montweiss} there exists a smooth family $Q_t$ (for $t$ in a right neighborhood of $0$) with the following properties:
 \begin{itemize}
  \item $Q_0=\Phi(P^*)$;
  \item $Q_t$ has $1$-skeleton $\Gamma^*$;
  \item the length of the edge of $Q_t$ corresponding to $e^*$ is equal to $t$.
 \end{itemize}
Then the family $\Phi^{-1}(Q_t^*)$ is smooth and clearly satisfies the first two requirements in the statement of this lemma. To verify the third, consider $F^t_1$ and $F^t_2$, and their dual points $v^t_1$ and $v^t_2$. Denote with $V_1^t$ and $V_2^t$ the points on the one-sheeted hyperboloid $H_S^3\subseteq \R^{1,3}$ that project down to $v_1^t$ and $v_2^t$ respectively; up to applying a suitable family of isometries $g_t$ and reparametrizing $Q_t$, we can assume that $v^t_1=\left(\frac{1}{\sqrt{\alpha^2-1}},\frac{\alpha}{\sqrt{\alpha^2-1}},0,0\right)\in\R^{1,3}$ and $v^t_2=\left(\frac{1}{\sqrt{\alpha^2+t^2-1}},\frac{\alpha}{\sqrt{\alpha^2+t^2-1}},\frac{t}{\sqrt{\alpha^2+t^2-1}},0\right)$. The angle between $F^t_1$ and $F_2^t$ is given by the formula
\begin{gather*}
 \theta_e^t=\textrm{ArcCos}\left(-\langle v_1^t,v_2^t\rangle\right)=\textrm{ArcCos}\left(\frac{1-\alpha^2}{\sqrt{\left(\left(\alpha^2-1\right)\left(\alpha^2+t^2-1\right)\right)}}\right)
\end{gather*}

The derivative of $\theta_e^t$ with respect to $t$ evaluated in $t=0$ can be computed to be
\begin{displaymath}
-\sqrt{\frac{1}{\alpha^2-1}}
\end{displaymath}
which is different from $0$ (notice that $\alpha=1$ would imply that $v_1^t$ is an ideal point which is not the case).
\end{proof}

\begin{teo}\label{prop:deform}
 The map $\Psi:\mathcal{A}_\Gamma^e\ra \R^n$ assigning to each polyhedron up to isometry its tuple of dihedral angles is a local diffeomorphism.
\end{teo}
\begin{proof}
 Consider a polyhedron $P\in\mathcal{A}_\Gamma^e$. If $P$ is in the interior of $\mathcal{A}_\Gamma^e$, then $\Psi$ is a local diffeomorphism thanks to \cite[Theorem 19]{mont}. Consider then $P\in\partial\mathcal{A}_\Gamma^e\cong \mathcal{A}_{\Gamma'}$ where $\Gamma'$ is obtained from $\Gamma$ by removing the edge $e$. Then by Corollary \ref{cor:manifold} we can choose $v_2,\dots,v_{n}\in T_P\mathcal{A}_{\Gamma'}$ and a vector $v_1\in T_P\mathcal{A}^e_\Gamma$ which is tangent to a family $P_t$ whose existence is guaranteed by Lemma \ref{lem:mainlem}. In this basis, 
 \begin{displaymath}
  d_P\Psi=\left(\begin{matrix}
 a & \begin{matrix} 0 & 0 & \dots & 0 \end{matrix} \\
 \begin{matrix} * \\ * \\ \vdots \\[1ex] * \end{matrix} & \text{\Large$d_P\Psi'$}
\end{matrix}\right)
 \end{displaymath}
where $\Psi'$ is the dihedral angle map of $\mathcal{A}_{\Gamma'}$ and $a<0$. Because $d_P\Psi'$ is invertible (by \cite[Proposition 17]{mont}) and $a\neq 0$ (by Lemma \ref{lem:mainlem}) this matrix is invertible, hence $\Psi$ is a local diffeomorphism at $P$.
\end{proof}

\section{The volume conjecture for hyperbolic polyhedra}\label{sec:quantum}

In this section we briefly talk about the Yokota invariant and the Volume Conjecture for polyhedra. More details about these topics can be found in \cite{maxvolconj}.

Take a graph $\Gamma$ contained in a manifold $M$, an odd integer $r>2$, a primitive $r$-th root of unity $q$ and a coloring $col$ of the edges of $\Gamma$ with natural even numbers less than $r-2$; then the Yokota invariant $Y_r(M,\Gamma,col,q)$ is defined in \cite{yok} as a direct generalization of the Kauffman skein bracket of a trivalent graph. 

The simplest case to which these invariants apply is that of a planar graph in $S^3$; in this case, if the graph is $3$-connected (that is to say, it cannot be disconnected by removing $2$ edges and has more than $3$ vertices) then it is the $1$-skeleton of a hyperbolic polyhedron. The Volume Conjecture for polyhedra relates the growth of the invariants $Y_r\left(S^3,\Gamma,col,e^{2\pi i/r}\right)$ (as $r\ra\infty$) to the volume of a hyperbolic polyhedron with $1$-skeleton $\Gamma$ and dihedral angles determined by $col$ in a certain way.

A conjecture relating the growth of the quantum $6j$-symbol to the volume of hyperbolic tetrahedra first appeared in \cite{C}. Later, various generalizations to the case of trivalent graphs and simple polyhedra were proposed in \cite{volconjpoly} and \cite{murkolp}; the following version of the conjecture first appeared in \cite{maxvolconj} as a direct generalization of these conjectures to non-simple polyhedra.

 \begin{cnj}[The volume conjecture for polyhedra]\label{cnj:volconj}
Let $P$ be a proper polyhedron with dihedral angles
  $\alpha_1,\dots,\alpha_m$ at the edges $e_1,\dots,e_m$, and $1$-skeleton $\Gamma$. Let $col_r$ be a sequence of $r$-admissible colorings of the edges $e_1,\dots,e_m$ of $\Gamma$ such that 
  \begin{displaymath}\label{cnd:angles}
   2\pi\lim_{r\ra+\infty}\frac{col_r(e_i)}{r}=\pi-\alpha_i.
  \end{displaymath}
Then
\begin{displaymath}
 \lim_{r\ra+\infty}\frac{\pi}{r}\log\left\lvert Y_r(S^3,\Gamma,col_r,e^{2\pi i/r})\right\rvert=\mathrm{Vol}(P).
\end{displaymath}

 \end{cnj}
 
 \begin{oss}
  Conjecture \ref{cnj:volconj} is stated for proper generalized hyperbolic polyhedra, that is to say, polyhedra with possibly hyperideal vertices and whose truncations are all right-angled; for the terminology see \cite{maxvol}. In the present paper we are only concerned with hyperbolic polyhedra as in Definition \ref{dfn:poly}, i.e. compact polyhedra. Therefore, it is safe to mentally substitute ``proper'' for ``hyperbolic'' in the statement of Conjecture \ref{cnj:volconj} (and any subsequent mention of the term).
 \end{oss}

 Conjecture \ref{cnj:volconj} holds (see \cite{chenmur}) for any polyhedron obtained by gluing tetrahedra with at least one hyperideal vertex along some of their truncation faces (in particular, some of these polyhedra are compact but all of them are simple). It was also proved for another, larger family of non-simple polyhedra in \cite{maxvolconj} (albeit for a single sequence of colors each). Further numerical evidence for it appears in the appendix of \cite{maxvolconj}.
 
 \begin{prop}
  If Conjecture \ref{cnj:volconj} is true, then the weak Stoker Conjecture is true.
 \end{prop}
 \begin{proof}
  If $P_1$ and $P_2$ have the same $1$-skeleton $\Gamma$ and the same dihedral angles $\theta_1,\dots,\theta_n$, then we can choose a sequence of colorings $col_r$ satisfying condition \ref{cnd:angles}. Then the Volume Conjecture implies that
  \begin{displaymath}
 \lim_{r\ra+\infty}\frac{\pi}{r}\log\left\lvert Y_r(\Gamma,col_r)\right\rvert=\mathrm{Vol}(P_1)=\mathrm{Vol}(P_2).
\end{displaymath}
 \end{proof}

\section{The weak Stoker conjecture implies the Stoker conjecture}\label{sec:main}

We now proceed to show that the weak version of the Stoker conjecture implies the strong version. The main tool involved, other than Theorem \ref{prop:deform}, is the Schl\"afli identity.

\begin{teo}\label{teo:schlafli}[\cite[Chapter: The Schl\"afli differential equality]{miln}]
 If $P_t$ is a smooth family of hyperbolic polyhedra with the same $1$-skeleton, dihedral angles $\theta_1^t,\dots,\theta_n^t$ and edge lengths $l_1^t,\dots,l_n^t$, then
 \begin{displaymath}
  \frac{\partial\vol(P_t)}{\partial t}_{\big\rvert_{t=t_0}}=-\frac{1}{2}\sum_i l_i^{t_0}\frac{\partial\theta^t_i}{\partial t}_{\big\rvert_{t=t_0}}
 \end{displaymath}
\end{teo}

\begin{prop}\label{prop:edgelengths}
 If Conjecture \ref{cnj:wstok} is true, then two polyhedra with the same $1$-skeleton and the same dihedral angles have the same edge lengths.
\end{prop}
\begin{proof}
 Consider $P_1$ and $P_2$ with the same $1$-skeleton $\Gamma$ and the same dihedral angles $\{\theta_e\}_{e \textrm{ edges of }\Gamma}$; fix an edge $e\subseteq \Gamma$. By Theorem \ref{thm:main}, for $t$ close to $1$ there are two families of polyhedra with $1$-skeleton $\Gamma$, denoted with $P_1^t$ and $P_2^t$, such that:
 \begin{itemize}
  \item $P_1^1=P_1$ and $P_2^1=P_2$;
  \item the dihedral angles of both $P_1^t$ and $P_2^t$ at $e$ are equal to $t\theta_e$.
  \item the dihedral angles of both $P_1^t$ and $P_2^t$ at an edge $e'\neq e$ are equal to $\theta_{e'}$.
 \end{itemize}
Because of the assumption of the validity of Conjecture \ref{cnj:wstok}, $\vol(P_1^t)=\vol(P_2^t)$ for all $t$ in a neighborhood of $1$; therefore by the Schl\"afli identity, $l_e^1=\frac{\partial\vol(P_1^t)}{\partial\theta_e}|_{t=1}=\frac{\partial\vol(P_2^t)}{\partial\theta_e}|_{t=1}=l_e^2$.
\end{proof}

\begin{prop}\label{prop:diaglengths}
  If Conjecture \ref{cnj:wstok} is true, then two polyhedra with the same $1$-skeleton and the same dihedral angles have congruent faces.
\end{prop}
\begin{proof}
 Consider two polyhedra $P_1$ and $P_2$ with the same $1$-skeleton $\Gamma$ and the same dihedral angles $\{\theta_e\}_{e \textrm{ edges of }\Gamma}$. Fix faces $F_1$ and $F_2$ (of $P_1$ and $P_2$ respectively) corresponding to the same face $F$ of $\Gamma$. If $F$ is triangular, then $F_1$ and $F_2$ have the same edge lengths by Proposition \ref{prop:edgelengths}, and therefore they are congruent. Suppose that $F$ is not triangular: we show that $F_1$ and $F_2$ are congruent by showing that their diagonals have the same length. Choose two non-adjacent vertices $v_1$ and $v_2$ of $F$, and add the edge $\tilde{e}=\overline{v_1v_2}$ to $\Gamma$; we call the new graph $\Gamma'$. Since $P_1$ and $P_2$ are both elements of $\mathcal{P}_\Gamma=\partial\mathcal{P}_{\Gamma'}^{\tilde{e}}$, we can think of them as polyhedra with weak $1$-skeleton $(\Gamma',\tilde{e})$; in particular their dihedral angle at $\tilde{e}$ is equal to $\pi$.
 
 By Theorem \ref{thm:main}, for $t$ in some interval $(1-\epsilon, 1]$ there are two families of polyhedra with $1$-skeleton $\Gamma'$, denoted with $P_1^t$ and $P_2^t$, such that:
 \begin{itemize}
  \item $P_1^1=P_1$ and $P_2^1=P_2$;
  \item the dihedral angles of both $P_1^t$ and $P_2^t$ at $\tilde{e}$ are equal to $t\pi$.
  \item the dihedral angles of both $P_1^t$ and $P_2^t$ at an edge $e\neq \tilde{e}$ are equal to $\theta_{e'}$.
 \end{itemize}
 
 Therefore, by the same argument of Proposition \ref{prop:edgelengths}, the length of $\tilde{e}$ is the same in $P_1$ and in $P_2$. This shows that every diagonal in $P_1$ has the same length of the corresponding diagonal in $P_2$, which implies that each face of $P_1$ is congruent to the corresponding face of $P_2$.
\end{proof}

\begin{cor}\label{cor:equiv}
 If Conjecture \ref{cnj:wstok} is true, then the Stoker conjecture is true.
\end{cor}
\begin{proof}
Let $P_1$ and $P_2$ have the same $1$-skeleton and dihedral angles. Assuming that Conjecture \ref{cnj:wstok} is true, then by Proposition \ref{prop:diaglengths}, $P_1$ and $P_2$ have congruent faces. Then, by \cite[Theorem 4.1]{rivhodg}, two polyhedra with the same $1$-skeleton and congruent faces are isometric.
\end{proof}

\bibliographystyle{alpha}
\bibliography{Bibliography}

\def\cprime{$'$}
\begin{thebibliography}{CGvdV15}

\bibitem[And70]{andreev}
E.~M. Andreev.
\newblock {On convex polyhedra in Lobachevskii spaces}.
\newblock {\em Mat. Sb.}, 123(3):445--478, 1970.

\bibitem[BB02]{bonbao}
X.~Bao and F.~Bonahon.
\newblock Hyperideal polyhedra in hyperbolic 3-space.
\newblock {\em Bull. Soc. Math. de France}, 130(3):457--491, 2002.

\bibitem[BDKY22]{BDKY2018}
G.~Belletti, R.~Detcherry, E.~Kalfagianni, and T.~Yang.
\newblock Growth of quantum $6 j $-symbols and applications to the volume
  conjecture.
\newblock {\em Journal of differential geometry}, 120(2):199--229, 2022.

\bibitem[Bel20]{maxvolconj}
G.~Belletti.
\newblock A maximum volume conjecture for hyperbolic polyhedra.
\newblock {\em arXiv:2002.01904}, 2020.

\bibitem[Bel21]{maxvol}
G.~Belletti.
\newblock The maximum volume of hyperbolic polyhedra.
\newblock {\em Transactions of the American Mathematical Society},
  374(2):1125--1153, 2021.

\bibitem[CGvdV15]{volconjpoly}
F.~Costantino, F.~Gu{\'e}ritaud, and R.~van~der Veen.
\newblock On the volume conjecture for polyhedra.
\newblock {\em Geom. Dedicata}, 179(1):385--409, 2015.

\bibitem[CM]{chenmur}
Q.~Chen and J.~Murakami.
\newblock {Asymptotics of Quantum 6$j$-Symbols}.
\newblock {\em arXiv preprint math.GT/1706.04887}.

\bibitem[Cos07]{C}
F.~Costantino.
\newblock {$6j$}-symbols, hyperbolic structures and the volume conjecture.
\newblock {\em Geom. Topol.}, 11:1831--1854, 2007.

\bibitem[CY18]{cyvolconj}
Q.~Chen and T.~Yang.
\newblock {Volume conjectures for the Reshetikhin--Turaev and the Turaev--Viro
  invariants}.
\newblock {\em Quantum Topol.}, 9(3):419--460, 2018.

\bibitem[Fle73]{fle}
H.~Fleischner.
\newblock The uniquely embeddable planar graphs.
\newblock {\em Discrete Math.}, 4(4):347--358, 1973.

\bibitem[Kas97]{kash}
R.~Kashaev.
\newblock The hyperbolic volume of knots from the quantum dilogarithm.
\newblock {\em Letters in Mathematical Physics}, 39(3):269--275, 1997.

\bibitem[KM18]{murkolp}
A.~Kolpakov and J.~Murakami.
\newblock {Combinatorial Decompositions, Kirillov--Reshetikhin Invariants, and
  the Volume Conjecture for Hyperbolic Polyhedra}.
\newblock {\em Experimental Mathematics}, 27(2):193--207, 2018.

\bibitem[Mil94]{miln}
J.W. Milnor.
\newblock {\em {Collected papers. 1. Geometry}}.
\newblock Publish or Perish, 1994.

\bibitem[MM01]{murmur}
H.~Murakami and J.~Murakami.
\newblock The colored jones polynomials and the simplicial volume of a knot.
\newblock {\em Acta Mathematica}, 186(1):85--104, 2001.

\bibitem[Mon13]{mont}
G.~Montcouquiol.
\newblock Deformations of hyperbolic convex polyhedra and cone-3-manifolds.
\newblock {\em Geom. Dedicata}, 166(1):163--183, 2013.

\bibitem[MW13]{montweiss}
G.~Montcouquiol and H.~Wei{\ss}.
\newblock Complex twist flows on surface group representations and the local
  shape of the deformation space of hyperbolic cone--3--manifolds.
\newblock {\em Geometry \& Topology}, 17(1):369--412, 2013.

\bibitem[Oht16]{ohtsuki52}
T.~Ohtsuki.
\newblock {On the asymptotic expansion of the Kashaev invariant of the $5_2$
  knot}.
\newblock {\em Quantum Topol}, 7(4):669--735, 2016.

\bibitem[Oht17]{ohtsuki7}
T.~Ohtsuki.
\newblock {On the asymptotic expansions of the Kashaev invariant of hyperbolic
  knots with seven crossings}.
\newblock {\em International Journal of Mathematics}, 28(13):1750096, 2017.

\bibitem[Oht18]{ohtdf}
T.~Ohtsuki.
\newblock {On the asymptotic expansion of the quantum $SU(2)$ invariant at
  $q=\exp(4\pi\sqrt{-1}/N)$ for closed hyperbolic $3$-manifolds obtained by
  integral surgery along the figure-eight knot}.
\newblock {\em Algebr. Geom. Topol.}, 18(7):4187--4274, 2018.

\bibitem[OY18]{ohtsuki6}
T.~Ohtsuki and Y.~Yokota.
\newblock {On the asymptotic expansions of the Kashaev invariant of the knots
  with 6 crossings}.
\newblock In {\em Mathematical Proceedings of the Cambridge Philosophical
  Society}, volume 165, pages 287--339. Cambridge University Press, 2018.

\bibitem[RH93]{rivhodg}
I.~Rivin and C.~Hodgson.
\newblock A characterization of compact convex polyhedra in hyperbolic 3-space.
\newblock {\em Invent. Math.}, 111(1):77--111, 1993.

\bibitem[Riv96]{rivin}
I.~Rivin.
\newblock A characterization of ideal polyhedra in hyperbolic 3-space.
\newblock {\em Ann. of Math.}, pages 51--70, 1996.

\bibitem[Ste22]{steinitz}
E.~Steinitz.
\newblock Polyeder und raumeinteilungen.
\newblock {\em Encyk. der Math. Wiss.}, 12:38--43, 1922.

\bibitem[Sto68]{stoker}
J.~Stoker.
\newblock Geometrical problems concerning polyhedra in the large.
\newblock {\em Communications on pure and applied mathematics}, 21(2):119--168,
  1968.

\bibitem[Wei13]{weiss}
H.~Weiss.
\newblock The deformation theory of hyperbolic cone--3--manifolds with
  cone-angles less than 2$\pi$.
\newblock {\em Geom. Topol.}, 17(1):329--367, 2013.

\bibitem[Wit89]{witten}
E.~Witten.
\newblock {Quantum field theory and the Jones polynomial}.
\newblock {\em Communications in Mathematical Physics}, 121(3):351--399, 1989.

\bibitem[Yok96]{yok}
Y.~Yokota.
\newblock Topological invariants of graphs in 3-space.
\newblock {\em Topology}, 35(1):77--87, 1996.

\end{thebibliography}

\address
\end{document}